\documentclass[12pt]{amsart}
\usepackage{amssymb,amsmath,amsthm,latexsym}
\usepackage{mathrsfs}
\usepackage{mdwlist}
\usepackage{graphicx}
\usepackage{a4wide}

\newtheorem*{proposition}{The main result}

\theoremstyle{definition}

\newtheorem*{remark}{Remark}

\theoremstyle{remark}

\numberwithin{equation}{section}

\newcounter{smallromans}

  {\end{list}}

\newcounter{smallalphs}

  {\end{list}}

\author{Tomasz Kania}
\address{Department of Mathematics and Statistics, Fylde College,
  Lancaster University, Lancaster LA1 4YF, United Kingdom}
\email{t.kania@lancaster.ac.uk}
\title[A reflexive space whose algebra of operators is not a Grothendieck space]{A reflexive Banach space whose algebra of operators is not a Grothendieck space}
\begin{document}
\maketitle
\begin{abstract}By a result of Johnson, the Banach space $F=\big(\bigoplus_{n=1}^\infty \ell_1^n \big)_{\ell_\infty}$ contains a complemented copy of $\ell_1$. We identify $F$ with a complemented subspace of the space of (bounded, linear) operators on the reflexive space $\big(\bigoplus_{n=1}^\infty \ell_1^n \big)_{\ell_p}$ ($p\in (1,\infty))$, thus solving negatively the problem posed in the monograph of Diestel and Uhl which asks whether the space of operators on a reflexive Banach space is Grothendieck.\end{abstract}
\section{Introduction}

A Banach space $E$ is \emph{Grothendieck} if weak* convergent sequences in $E^*$ converge weakly. Certainly, every reflexive Banach space is Grothendieck. Notable examples of non-reflexive Grothendieck spaces are $C(K)$-spaces for extermally disconnected compact spaces $K$ (\cite{grothendieck}) and the Hardy space $H^\infty$ of bounded holomorphic functions on the unit disc (\cite{bourgain}).
Diestel and Uhl wrote in their famous monograph \cite[p.~180]{diesteluhl}:
\begin{quotation}\emph{Finally, there is some evidence (Akemann [1967], [1968]) that the space $\mathscr{L}(H;H)$ of bounded linear operators on a Hilbert space is a Grothendieck space and that more generally the space $\mathscr{L}(X;X)$ is a Grothendieck space for any reflexive Banach space $X$.}\end{quotation}
The question of whether the space of (bounded, linear) operators on a reflexive Banach space is Grothendieck was raised also by Soyba\c{s} (\cite{soybas}). Pfitzner proved in \cite{pfitzner} that C*-algebras have the so-called \emph{Pe\l czy\'{n}ski's property} (V) which for dual Banach spaces is equivalent to being a Grothendieck space (cf.~\cite[Exercise 12, p.~116]{diestel}). In particular, von Neumann algebras are Grothendieck spaces which confirms that the space of operators on a Hilbert space is Grothendieck. It is known that duals of spaces with property (V) are weakly sequentially complete. We shall present an example of a reflexive Banach space $E$ such that $\mathscr{B}(E)$ fails to be Grothendieck, giving thus a negative answer to the above-mentioned problem. To do this, we require a result of Johnson which asserts that the Banach space $F=\big(\bigoplus_{n=1}^\infty \ell_1^n \big)_{\ell_\infty}$ contains a complemented copy of $\ell_1$ (cf.~the Remark after Theorem 1 in \cite{johnson}), so it is not a Grothendieck space.

By an \emph{operator} we understand a bounded, linear operator acting between Banach spaces. The space $\mathscr{B}(E_1, E_2)$ of operators acting between spaces $E_1$ and $E_2$ is a Banach space  when endowed with the operator norm. We write $\mathscr{B}(E)$ for $\mathscr{B}(E,E)$. Let $p\in [1,\infty]$. We denote by $(\bigoplus_{n=1}^\infty E_n)_{\ell_p}$ the $\ell_p$-sum of a sequence $(E_n)_{n=1}^\infty$ of Banach spaces. We identify elements of $\mathscr{B}((\bigoplus_{n=1}^\infty E_n)_{\ell_p})$ with \emph{matrices} $(T_{ij})_{i,j\in \mathbb{N}}$, where $T_{ij}\in \mathscr{B}(E_j, E_i)$ ($i,j\in \mathbb{N}$). Let $(e_n)_{n=1}^\infty$ be the canonical basis of $\ell_1$. For each $n\in \mathbb{N}$ we define $\ell_1^n = \mbox{span}\{e_1, \ldots, e_n\}$.

\section{The result}
\begin{proposition}Let $p\in (1,\infty)$ and consider the reflexive Banach space $E = \big(\bigoplus_{n=1}^\infty \ell_1^n \big)_{\ell_p}$. Then $\mathscr{B}(E)$ is not a Grothendieck space.\end{proposition}
\begin{proof}Recall that $F=\big(\bigoplus_{n=1}^\infty \ell_1^n \big)_{\ell_\infty}$ contains a complemented copy of $\ell_1$. To complete the proof it is enough to embed $F$ as a complemented subspace of $\mathscr{B}(E)$.

One may identify $\ell_1^n$ with a 1-complemented subspace of $\mathscr{B}(\ell_1^n)$ via the mapping 
\[e_k\mapsto e_k\otimes e_1^*\;\; (k\leqslant n,\, n\in \mathbb{N}),\] 
where $e_1^*$ stands for the coordinate functional associated with $e_1$. Consequently, the space $D=(\bigoplus_{n=1}^\infty \mathscr{B}(\ell_1^n))_{\ell_\infty}$ contains a complemented subspace isomorphic to $F$. Let $\Delta\colon D\to \mathscr{B}(E)$ be the \emph{diagonal embedding}, that is, $\Delta( (T_n)_{n=1}^\infty ) = \mbox{diag}(T_1, T_2, \ldots)$ ($(T_n)_{n=1}^\infty \in D$); this map is well-defined since the decomposition of $E$ into the subspaces $\ell_1^1, \ell_1^2,\ldots $ is unconditional.

It is enough to notice that $\Delta$ has a left-inverse $\Xi\colon \mathscr{B}(E)\to D$ given by 
\[\Xi (T_{ij})_{i,j\in \mathbb{N}} = (T_{ii})_{i=1}^\infty\;\;((T_{ij})_{i,j\in\mathbb{N}}\in \mathscr{B}(E)),\] 
which is bounded. To this end, we shall perform a construction inspired by a trick of Tong (\emph{cf.}~\cite[Theorem 2.3]{tong} and its proof). With each operator $T=(T_{ij})_{i,j\in\mathbb{N}}\in \mathscr{B}(E)$ we shall associate a sequence $\left(S^{(n)}\right)_{n=1}^\infty$ of finite-rank perturbations of $T$ such that for each $n\in\mathbb{N}$ we have $\|S^{(n)}\|\leqslant \|T\|$ and the matrix of $S^{(n)}$ agrees with the matrix of the diagonal operator $\mbox{diag}(-T_{11}, \ldots, -T_{nn}, 0, 0, \ldots)$ at entries $(i,j)$ with $i\leqslant n$ or $j\leqslant n$. This will immediately yield that
\[\|\Xi (T)\| =  \sup_{n\in \mathbb{N}}\|T_{nn}\| = \sup_{n\in \mathbb{N}}\|-S^{(n)}_{nn}\| \leqslant \sup_{n\in \mathbb{N}}\|S^{(n)}\| \leqslant \|T\|.\]
Define operators $T_{\mathsf k}, T_{\mathsf r}$ which have the same columns and rows as $T$ respectively, except the first ones, where we instead set $(T_{\mathsf k})_{i1} = - T_{i1}$ and $(T_{\mathsf r})_{1j} = - T_{1j}$ for $i,j\in \mathbb{N}$ (these are indeed elements of $\mathscr{B}(E)$ as rank-1 perturbations of $T$). Certainly, $\|T\| = \|T_{\mathsf k}\| =  \|T_{\mathsf r}\|$ and the norm of $S= ( T_{\mathsf k}+T_{\mathsf r} )/2$ does not exceed the norm of $T$. Arguing similarly, we observe that $\| ( S^{(n)}_{\mathsf k}+S^{(n)}_{\mathsf r} )/2\|\leqslant \|T\|$, where $S^{(1)}=S$ and $S^{(n+1)} =( S^{(n)}_{\mathsf k}+S^{(n)}_{\mathsf r} )/2$ ($n\in\mathbb{N}$). Consequently, $\left(S^{(n)}\right)_{n=1}^\infty$ is the desired sequence.
 \end{proof}

\begin{remark}The space $\mathscr{B}(E)$ shares with the space of operators on a Hilbert space a number of common properties. For instance, since $E$ has a Schauder basis, $\mathscr{B}(E)$ can be identified with the bidual of $\mathscr{K}(E)$, the space of compact operators on $E$. Nonetheless, $E$ is plainly not superreflexive and $\mathscr{B}(E)$ fails to have weakly sequentially dual for the obvious reason that $\ell_\infty$ embeds into $\mathscr{B}(E)^*$. We conjecture that the space of operators on a superreflexive space is Grothendieck (or at least it has weakly sequentially complete dual).\end{remark}

\end{document}